\documentclass{my-elsarticle}

\usepackage{array}
\usepackage{booktabs}
\usepackage{amssymb,amsmath,amsthm}
\usepackage{epsfig,multicol,graphicx,epstopdf}
\usepackage{a4wide}

\newtheorem{example}{Example}[section]
\newtheorem{definition}{Definition}[section]
\newtheorem{theorem}{Theorem}[section]

\allowdisplaybreaks

\begin{document}

\begin{frontmatter}

\title{A numerical approach for solving fractional optimal control problems using modified hat functions}

\author[Add:a,Add:b]{Somayeh Nemati} 
\ead{s.nemati@umz.ac.ir, s.nemati@ua.pt}

\author[Add:c]{Pedro M. Lima}
\ead{pedro.t.lima@ist.utl.pt}

\author[Add:b]{Delfim F. M. Torres\corref{corA}} 
\ead{delfim@ua.pt}
\cortext[corA]{Corresponding author.}


\address[Add:a]{Department of Mathematics, 
Faculty of Mathematical Sciences, 
University of Mazandaran, Babolsar, Iran}

\address[Add:b]{Center for Research and Development in Mathematics and Applications (CIDMA),\\ 
Department of Mathematics, University of Aveiro, Aveiro 3810-193, Portugal}

\address[Add:c]{Centro de Matem\'{a}tica Computacional e Estoc\'astica, 
Instituto Superior T\'{e}cnico,\\ 
Universidade de Lisboa, Av. Rovisco Pais, 1049-001 Lisboa, Portugal}


\begin{abstract}
We introduce a numerical method, based on modified hat functions,  
for solving a class of fractional optimal control problems. 
In our scheme, the control and the fractional derivative of the
state function are considered as linear combinations of the modified hat functions. 
The fractional derivative is considered in the Caputo sense while the 
Riemann--Liouville integral operator is used to give approximations 
for the state function and some of its derivatives. To this aim, we use 
the fractional order integration operational matrix of the modified 
hat functions and some properties of the Caputo derivative and  
Riemann--Liouville integral operators. Using results 
of the considered basis functions, solving the fractional optimal 
control problem is reduced to the solution of a system of nonlinear 
algebraic equations. An error bound is proved for the approximate 
optimal value of the performance index obtained by the proposed method.  
The method is then generalized for solving a class of 
fractional optimal control problems with inequality constraints. 
The most important advantages of our method are easy implementation, 
simple operations, and elimination of numerical integration. 
Some illustrative examples are considered to demonstrate 
the effectiveness and accuracy of the proposed technique.
\end{abstract}

\end{frontmatter}


\section{Introduction}
\label{sec:1}

Optimal control theory allows us to choose control functions 
in a dynamical system to achieve a certain goal. The theory 
generalizes that of the classical calculus of variations 
and has found many applications in engineering, physics, economics, 
and life sciences \cite{MR3644954,MR3821582}. In 1996/97, 
motivated by nonconservative physical processes in mechanics,
the subject was extended to the case in which derivatives and 
integrals are understood as fractional operators of arbitrary order 
\cite{MyID:408,MR3221831}. 

Fractional optimal control is nowadays an important research area, 
allowing to apply the power of variational methods to real systems 
\cite{MR3443073,MR3822307,MyID:420}. It is an expanding branch of applied mathematics 
that attracts more and more attention, having been applied in transportation, electronic, 
chemical and biological systems \cite{MR3933377}.
Another applications of fractional optimal control problems are shown in \cite{MR3654793}.
Besides providing nontrivial generalizations of the classical theory, it opens doors to
new and interesting modern scientific problems 
\cite{MR3904404,Salati}. Given the complexity of the problems, 
numerical methods are indispensable \cite{MR3907974,MyID:426,MR3872489}. 

Recently, numerical methods based on modified hat functions have shown to
provide an effective way for solving fractional differential equations 
\cite{Nemati1,Nemati2}. Here, for the first time in the literature, 
we introduce such a method for solving a general class of fractional 
optimal control problems. More precisely, in this work we begin by 
considering the following fractional optimal control problem (FOCP):
\begin{equation}
\label{1.1}
\min~J=\int_0^{t_f}f(t,x(t),u(t))dt
\end{equation}
subject to the fractional dynamic system
\begin{equation}
\label{1.2}
D^\alpha x(t)=g\left(t,x(t),D^{\alpha_1}x(t),
D^{\alpha_2}x(t),\ldots ,D^{\alpha_k}x(t),u(t)\right),
\quad m-1<\alpha\leq m,\\
\end{equation}
and the initial conditions
\begin{equation}
\label{1.3}
x^{(i)}(0)=q_i,\quad i=0, 1, \ldots, m-1,
\end{equation}
where $\alpha>\alpha_k>\ldots>\alpha_2 >\alpha_1>0$, $D^{\alpha}$ 
denotes the (left) Caputo derivative of order $\alpha$, $x(\cdot)$ 
is the state function, $u(\cdot)$ is the control function, 
$J$ is called the performance index, and $f$ and $g$ are given 
continuous functions. The aim of problem \eqref{1.1}--\eqref{1.3} 
consists in finding the optimal state and control functions that minimize 
the performance index. In this paper, a numerical method based 
on the modified hat functions is introduced for solving problem 
\eqref{1.1}--\eqref{1.3} and then the method is developed 
for FOCPs with inequality constraints. Because of the importance 
of fractional optimal control in applications, 
the numerical solution of FOCPs is crucial
and our results can be very useful for many researchers.

The paper is organized as follows. In Section~\ref{sec:2}, 
some preliminaries of fractional calculus are reviewed 
and the basis functions and their properties are presented. 
A numerical method for solving the FOCP \eqref{1.1}--\eqref{1.3} 
is introduced in Section~\ref{sec:3}. In Section~\ref{sec:4}, 
an error estimate is proved for the proposed method. 
Section~\ref{sec:5} is devoted to the development 
of the method to solve constrained FOCPs. In Section~\ref{sec:6}, 
some FOCPs are considered and solved using the proposed method. 
Finally, a conclusion is given in Section~\ref{sec:7}.


\section{Preliminaries}
\label{sec:2}

In this section, a brief review on necessary definitions 
and properties of fractional calculus is presented. Furthermore, 
the modified hat functions and some of their properties are recalled.


\subsection{Preliminaries of fractional calculus}

We use fractional derivatives in the sense of Caputo.
They are defined via Riemann--Liouville fractional integral.

\begin{definition}[See, e.g., \cite{Podlubny}]
The Riemann--Liouville fractional integral operator with order 
$\alpha\geq 0$ of a given function $x(\cdot)$ is defined as 
\begin{equation*}
\begin{split}
&I^{\alpha}x(t)=\frac{1}{\Gamma(\alpha)}\int_0^t(t-s)^{\alpha-1}x(s)ds,\\
&I^0x(t)=x(t),
\end{split}
\end{equation*}
where $\Gamma(\cdot)$ is the Euler gamma function.
\end{definition}

\begin{definition}[See, e.g., \cite{Podlubny}]
The Caputo fractional derivative of order $\alpha>0$ 
of a function $x(\cdot)$ is defined as
\begin{equation*}
D^{\alpha}x(t)=\frac{1}{\Gamma(m-\alpha)}
\int_0^t(t-s)^{m-\alpha-1}x^{(m)}(s)ds,\quad m-1<\alpha \leq m.
\end{equation*}
\end{definition}

For $m-1<\alpha\leq m$, $m\in\mathbb{N}$, some useful properties of the 
Caputo derivative and Riemann--Liouville integral are as follows:
\begin{equation*}
D^\alpha I^\alpha x(t)=x(t),
\end{equation*}
\begin{equation}
\label{2.1}
I^\alpha D^\alpha x(t)=x(t)-\sum_{i=0}^{m-1}x^{(i)}(0)\frac{t^{i}}{i!},
\quad t>0,
\end{equation}
\begin{equation}
\label{2.2}
I^{\alpha-\beta}D^{\alpha}x(t)=D^\beta x(t)
-\sum_{i=\lceil\beta\rceil}^{m-1}x^{(i)}(0)
\frac{t^{i-\beta}}{\Gamma(i-\beta+1)},\quad 0< \beta<\alpha,
\quad t>0,
\end{equation}
where $\lceil \cdot \rceil$ is the ceiling function.


\subsection{Properties of the modified hat functions}

By considering an even integer number $n$, the interval $[0,t_f]$ is divided 
into $n$ subintervals $[ih,(i+1)h]$, $i=0,1,2,\ldots,n-1$, with $h=\frac{t_f}{n}$. 
Then the generalized modified hat functions consist of a set of $n+1$ linearly 
independent functions in $L^2[0,t_f]$ that are defined as follows 
\cite{Nemati1,Nemati2}:
\begin{equation*}
\psi_0(t)=\left\{
\begin{array}{ll}
\frac{1}{2h^2}(t-h)(t-2h),&0\leq t\leq 2h,\\
&\\
0,&\text{otherwise};
\end{array}\right.
\end{equation*}
if $i$ is odd and $1\leq i\leq n-1$,
\begin{equation*}
\psi_i(t)=\left\{
\begin{array}{ll}
\frac{-1}{h^2}(t-(i-1)h)(t-(i+1)h),&(i-1)h\leq t\leq(i+1)h,\\
&\\
0,&\text{otherwise};
\end{array}\right.
\end{equation*}
if $i$ is even and $2\leq i\leq n-2$,
\begin{equation*}
\psi_i(t)
=\left\{
\begin{array}{ll}
\frac{1}{2h^2}(t-(i-1)h)(t-(i-2)h),&(i-2)h\leq t\leq ih,\\
&\\
\frac{1}{2h^2}(t-(i+1)h)(t-(i+2)h),&ih\leq t\leq(i+2)h,\\
&\\
0,&\text{otherwise},
\end{array}\right.
\end{equation*}
and
\begin{equation*}
\psi_n(t)=\left\{
\begin{array}{ll}
\frac{1}{2h^2}(t-(t_f-h))(t-(t_f-2h)),&t_f-2h\leq t\leq t_f,\\
&\\
0,&\text{otherwise}.
\end{array}\right.
\end{equation*}
These functions satisfy some interesting properties, which make 
them very useful. Some of these properties are as follows:
\begin{equation*}
\psi_i(jh)=\left\{\begin{array}{ll}1,&i=j,\\
&\\
0,&i\neq j,
\end{array}\right.
\end{equation*}
\begin{equation*}
\sum_{i=0}^n \psi_i(t)=1,
\end{equation*}
\begin{equation}
\label{2.3}
\int_0^{t_f}\psi_i(t)dt
=\left\{
\begin{array}{ll}
\frac{h}{3},&\text{if $i=0,n$},\\
&\\
\frac{4h}{3}&\text{if $i$ is odd and $1\leq i\leq n-1$},\\
&\\
\frac{2h}{3},&\text{if $i$ is even and $2\leq i\leq n-2$}.
\end{array}\right.
\end{equation}

Any function $x(\cdot)\in L^2[0,t_f]$ can be approximated 
using the modified hat functions as
\begin{equation}
\label{2.4}
x(t)\simeq x_n(t)=\sum_{i=0}^n a_i \psi_i(t)=A^T \Psi(t),
\end{equation}
where
\begin{equation}
\label{2.5}
\Psi(t)=[\psi_0(t),\psi_1(t),\ldots,\psi_n(t)]^T,
\end{equation}
and
\begin{equation*}
A=[a_0,a_1,\ldots,a_n]^T,
\end{equation*}
with $a_i=x(ih)$.

Let $\Psi(\cdot)$ be the modified hat functions basis vector 
given by \eqref{2.5} and $\alpha>0$. Then,
\begin{equation}
\label{2.6}
I^\alpha\Psi(t)\simeq P^{(\alpha)}\Psi(t),
\end{equation}
where $P^{(\alpha)}$ is a matrix of dimension $(n+1)\times(n+1)$  
called the operational matrix of fractional integration of order $\alpha$. 
This matrix is given (see \cite{Nemati1,Nemati2}) as  
\begin{equation}
\label{2.7}
P^{(\alpha)}
=\frac{h^\alpha}{2\Gamma(\alpha+3)}\left[
\begin{array}{cccccccc}
0 & \beta_1  &\beta_2& \beta_3 &\beta_4 &\ldots & \beta_{n-1}& \beta_n \\
0 & \eta_0 & \eta_1 & \eta_2&\eta_3&\ldots &\eta_{n-2} & \eta_{n-1} \\
0 & \xi_{-1} &\xi_0 & \xi_1&\xi_2&\ldots &\xi_{n-3}& \xi_{n-2}\\
0 & 0 & 0 & \eta_0 & \eta_1 & \ldots &\eta_{n-4} & \eta_{n-3}\\
0 & 0 & 0 &\xi_{-1} &\xi_0 & \ldots &\xi_{n-5} & \xi_{n-4}\\
\vdots & \vdots &\vdots& \vdots& \vdots  &  &\vdots  & \vdots\\
0 & 0 & 0 & 0&0&\ldots &\eta_0 & \eta_1 \\
0 & 0 & 0 &0 &0 & \ldots &\xi_{-1} & \xi_0\\
\end{array}
\right]
\end{equation}
with
\begin{equation*}
\begin{split}
&\beta_1=\alpha(3+2\alpha),\\
&\beta_i=i^{\alpha+1}(2i-6-3\alpha)+2i^\alpha(1+\alpha)(2+\alpha)
-(i-2)^{\alpha+1}(2i-2+\alpha), \quad i=2,3,\ldots,n,\\
&\eta_0=4(1+\alpha),\\
&\eta_i=4[(i-1)^{\alpha+1}(i+1+\alpha)-(i+1)^{\alpha+1}(i-1-\alpha)],~~i=1,2,\ldots,n-1,\\
&\xi_{-1}=-\alpha,\\
&\xi_0=2^{\alpha+1}(2-\alpha),\\
&\xi_1=3^{\alpha+1}(4-\alpha)-6(2+\alpha),\\
&\xi_i=(i+2)^{\alpha+1}(2i+2-\alpha)-6i^{\alpha+1}(2+\alpha)-(i-2)^{\alpha+1}(2i-2+\alpha),
\quad i=2,3,\ldots,n-2.
\end{split}
\end{equation*}


\section{Numerical solution of problem \eqref{1.1}--\eqref{1.3}}
\label{sec:3}

In this section, we introduce a new numerical scheme  
for solving the optimal control problem \eqref{1.1}--\eqref{1.3}.
Our method is based on the properties of the modified hat functions.
We begin by expanding the variable $t$ in terms of the modified hat functions. 
Using \eqref{2.4}, this variable can then be rewritten as
\begin{equation}
\label{3.1}
t=\Theta^T\Psi(t),
\end{equation}
where
\begin{equation}
\Theta=[0,h,2h,\ldots ,t_f]^T.
\end{equation}
Let us suppose that
\begin{equation}
\label{3.2}
D^{\alpha}x(t)=A^T\Psi(t)
\end{equation}
and
\begin{equation}
\label{3.3}
u(t)\simeq u_n(t)=U^T\Psi(t),
\end{equation}
where
\begin{equation*}
A=[a_0,a_1,\ldots,a_n]^T,\quad U=[u_0,u_1,\ldots,u_n]^T,
\end{equation*}
such that the elements of $A$ and $U$ are unknown. By utilizing \eqref{2.2}, 
\eqref{2.6}, \eqref{3.2}, and the initial conditions given in \eqref{1.3}, 
for $s=1,2,\ldots,k$, we obtain that
\begin{equation}
\label{3.4}
\begin{split}
D^{\alpha_s}x(t)
&=I^{\alpha-\alpha_s}(D^{\alpha}x(t))
+\sum_{i=\lceil\alpha_s\rceil}^{m-1}x^{(i)}(0)
\frac{t^{i-\alpha_s}}{\Gamma(j-\alpha_s+1)}\\
&\simeq A^TI^{\alpha-\alpha_s}\Psi(t)
+\sum_{i=\lceil\alpha_s\rceil}^{m-1}
q_i\frac{t^{i-\alpha_s}}{\Gamma(i-\alpha_s+1)}\\
&\simeq (A^TP^{(\alpha-\alpha_s)}+D_s^T)\Psi(t)\\
&=X_s^T\Psi(t),
\end{split}
\end{equation}
where we have used the notations
\begin{equation*}
\sum_{i=\lceil\alpha_s\rceil}^{m-1}
q_i\frac{t^{i-\alpha_s}}{\Gamma(i-\alpha_s+1)}
\simeq D_s^T\Psi(t),\quad X_s=(A^TP^{(\alpha-\alpha_s)}+D_s^T)^T
\end{equation*}
with
\begin{equation*}
X_s=[x_{s,0},x_{s,1},\ldots, x_{s,n}]^T.
\end{equation*}
Furthermore, the state function can be approximated 
using \eqref{2.1}, \eqref{2.6}, and \eqref{3.2}, as
\begin{equation}
\label{3.5}
x(t)\simeq x_n(t)=X^T\Psi(t),
\end{equation}
where we have used
\begin{equation}
\label{3.10}
\sum_{i=0}^{m-1}q_i\frac{t^{i}}{i!}
\simeq D^T\Psi(t),\quad X=(A^TP^{(\alpha)}+D^T)^T
\end{equation}
with
\begin{equation*}
X=[x_{0},x_{1},\ldots,x_{n}]^T.
\end{equation*}
We rewrite the elements of the vectors $X$ and $X_s$, 
$s=1,2,\ldots,k$, in terms of the elements of the vector $A$. 
To this aim, suppose that
\begin{equation*}
\begin{split}
&P^{(\alpha-\alpha_s)}=[p_{ij}^s]_{i,j=0,1,\ldots,n},
\quad \quad s=1,2,\ldots,k,\\
&P^{(\alpha)}=[p_{ij}]_{i,j=0,1,\ldots,n}.
\end{split}
\end{equation*}
Taking into consideration \eqref{2.7}, we have
\begin{equation*}
\begin{split}
&p_{i0}=p_{i0}^{s}=0,\quad \text{for $i=0,1,2,\ldots,n$},\\
&p_{ij}=p_{ij}^{s}=0,\quad \text{for $j=1,3,5,\ldots,n-1,$ $i=j+2,j+3,\ldots,n$},\\
&p_{ij}=p_{ij}^{s}=0,\quad \text{for $j=2,4,6,\ldots,n,$ $i=j+1,j+2,\ldots,n$}.
\end{split}
\end{equation*}
Therefore, the entries $x_{j}$ and $x_{s,j}$ 
can be written in details as
\begin{equation*}
\begin{split}
&x_{0}=q_0,\\
&x_{j}=\sum_{i=0}^{j+1}a_ip_{ij}+\sum_{i=0}^{m-1}
q_i\frac{(jh)^{i}}{i!},\quad j=1,3,\ldots, n-1,\\
&x_{j}=\sum_{i=0}^{j}a_ip_{ij}+\sum_{i=0}^{m-1}
q_i\frac{(jh)^{i}}{i!},\quad j=2,4,\ldots, n,\\
\end{split}
\end{equation*}
\begin{equation*}
\begin{split}
&x_{s,j}=\sum_{i=0}^{j+1}a_ip_{ij}^s+\sum_{i=\lceil\alpha_s\rceil}^{m-1}
q_i\frac{(jh)^{i-\alpha_s}}{\Gamma(i-\alpha_s+1)},\quad j=1,3,\ldots, n-1,\\
&x_{s,j}=\sum_{i=0}^{j}a_ip_{ij}^s+\sum_{i=\lceil\alpha_s\rceil}^{m-1}
q_i\frac{(jh)^{i-\alpha_s}}{\Gamma(i-\alpha_s+1)},\quad j=2,4,\ldots, n.\\
\end{split}
\end{equation*}
We employ the approximations given above in order to obtain 
an approximation of the performance index. For the function 
$f:\mathbb{R}^3\rightarrow \mathbb{R}$, using \eqref{2.4}, 
\eqref{3.1}, \eqref{3.3}, and \eqref{3.5}, we can write
\begin{equation}
\label{3.6}
\begin{split}
f(t,x(t),u(t))&\simeq \sum_{i=0}^n f(ih,x(ih),u(ih))\psi_i(t)\\
&\simeq \sum_{i=0}^n f(ih,x_{i},u_i)\psi_i(t)\\
&=f(\Theta,X,U)\Psi(t),\\
\end{split}
\end{equation}
where
\begin{equation*}
f(\Theta,X,U)=[f(0,x_{0},u_0),f(h,x_{1},u_1),\ldots,f(t_f,x_{n},u_n)].
\end{equation*}
By substituting \eqref{3.6} into the performance index and employing 
\eqref{2.3}, we obtain that
\begin{equation}
\label{index}
J\simeq\int_0^{t_f}f(\Theta,X,U)\Psi(t)dt
=f(\Theta,X,U)\int_0^{t_f}\Psi(t)dt=f(\Theta,X,U)L,
\end{equation}
where
\begin{equation}
L=\frac{h}{3}\left[1,4,2,4,2,\ldots ,2,4,1\right]^T.
\end{equation}
In a similar way, taking into account \eqref{2.4}, \eqref{3.1}, 
and \eqref{3.3}--\eqref{3.5} for the function 
$g:\mathbb{R}^{k+3}\rightarrow \mathbb{R}$, we have
\begin{equation}
\label{3.7}
\begin{split}
g(t,x(t),D^{\alpha_1}x(t),\ldots ,D^{\alpha_k}x(t),u(t))
&\simeq \sum_{i=0}^n g\left(ih,x(ih),D^{\alpha_1}x(ih),\right.\\
&\left.\quad \ldots, D^{\alpha_k}x(ih),u(ih)\right)\psi_i(t)\\
&\simeq \sum_{i=0}^n g\left(ih,x_{i},x_{1,i},
\ldots, x_{k,i},u_i\right)\psi_i(t)\\
&=g(\Theta,X,X_1,\ldots, X_k,U)\Psi(t),
\end{split}
\end{equation}
where
\begin{multline*}
g\left(\Theta,X,X_1,\ldots, X_k,U\right)\\
=\left[g(0,x_{0},x_{1,0},\ldots,
x_{k,0},u_0),g(h,x_{1},x_{1,1},\ldots,x_{k,1},u_1),
\ldots,g\left(t_f,x_{n},x_{1,n},\ldots,x_{k,n},u_n\right)\right].
\end{multline*}
Hence, using \eqref{3.2} and \eqref{3.7}, the dynamic 
system given by \eqref{1.2} is reduced to
\begin{equation}
\label{3.8}
A^T-g(\Theta,X,X_1,\ldots, X_k,U)=0.
\end{equation}
Finally, according to the Lagrange multiplier method 
for minimizing \eqref{index} subject to the conditions 
given in \eqref{3.8}, we define:
\begin{equation*}
J^*[A,U,\lambda]=f(\Theta,X,U)L
+\left(A^T-g(\Theta,X,X_1,\ldots, X_k,U)\right)\lambda,
\end{equation*}
where
\begin{equation*}
\lambda=[\lambda_0,\lambda_1,\ldots,\lambda_n]^T
\end{equation*}
is the unknown Lagrange multiplier. 
The necessary optimality conditions are as follows:
\begin{equation}
\label{3.9}
\frac{\partial J^*}{\partial A}=0,
\quad \frac{\partial J^*}{\partial U}=0, 
\quad \frac{\partial J^*}{\partial \lambda}=0.
\end{equation}
System \eqref{3.9} includes $3(n+1)$ nonlinear equations with $3(n+1)$ 
unknown parameters, which are the elements of $A$, $U$, and $\lambda$. 
By solving this system, approximations of the optimal control and 
corresponding state functions are given by \eqref{3.3} and \eqref{3.5}, 
respectively. Furthermore, an approximation of the optimal value 
of the performance index is given by \eqref{index}. The speed of our 
method depends on the speed of solving system \eqref{3.9}. Therefore, 
the form of the system is very important. Here, we detail the closed 
forms for system \eqref{3.9}. For simplicity, suppose that 
$g:=g(t,x(t),u(t))$. Then, we have
\begin{equation*}
\begin{split}
\frac{\partial J^*}{\partial{a_0}}
&=\frac{4h}{3}\sum_{j=0}^{\frac{n}{2}-1}p_{0(2j+1)}
f_{x_{2j+1}}((2j+1)h,x_{2j+1},u_{2j+1})+\frac{2h}{3}
\sum_{j=1}^{\frac{n}{2}-1}p_{0(2j)}f_{x_{2j}}(2jh,x_{2j},u_{2j})\\
&\quad +\frac{h}{3}p_{0n}f_{x_n}(t_f,x_n,u_n)+\lambda_0
-\sum_{j=1}^n\lambda_j p_{0j}g_{x_j}(jh,x_j,u_j)=0,
\end{split}
\end{equation*}
\begin{equation*}
\begin{split}
\frac{\partial J^*}{\partial {a_i}}
&=\frac{4h}{3}\sum_{j=\frac{i-1}{2}}^{\frac{n}{2}-1}p_{i(2j+1)}
f_{x_{2j+1}}((2j+1)h,x_{2j+1},u_{2j+1})+\frac{2h}{3}
\sum_{j=\frac{i+1}{2}}^{\frac{n}{2}-1}p_{i(2j)}f_{x_{2j}}(2jh,x_{2j},u_{2j})\\
&\quad +\frac{h}{3}p_{in}f_{x_n}(t_f,x_n,u_n)+\lambda_i
-\sum_{j=i}^n\lambda_j p_{ij}g_{x_j}(jh,x_j,u_j)=0,
\quad i=1,3,\ldots,n-1,
\end{split}
\end{equation*}
\begin{equation*}
\begin{split}
\frac{\partial J^*}{\partial {a_i}}
&=\frac{4h}{3}\sum_{j=\frac{i}{2}-1}^{\frac{n}{2}-1}p_{i(2j+1)}
f_{x_{2j+1}}((2j+1)h,x_{2j+1},u_{2j+1})+\frac{2h}{3}
\sum_{j=\frac{i}{2}}^{\frac{n}{2}-1}p_{i(2j)}
f_{x_{2j}}(2jh,x_{2j},u_{2j})\\
&+\frac{h}{3}p_{in}f_{x_n}(t_f,x_n,u_n)+\lambda_i
-\sum_{j=i-1}^n\lambda_j p_{ij}g_{x_j}(jh,x_j,u_j)=0,
\quad i=2,4,\ldots,n,
\end{split}
\end{equation*}
\begin{equation*}
\frac{\partial J^*}{\partial {u_i}}
=\frac{h}{3}f_{u_i}(ih,x_i,u_i)
-\lambda_i g_{u_i}(ih,x_i,u_i)=0,
\quad i=0,n,
\end{equation*}
\begin{equation*}
\frac{\partial J^*}{\partial {u_i}}
=\frac{4h}{3}f_{u_i}(ih,x_i,u_i)
-\lambda_i g_{u_i}(ih,x_i,u_i)=0,
\quad i=1,3,\ldots,n-1,
\end{equation*}
\begin{equation*}
\frac{\partial J^*}{\partial {u_i}}
=\frac{2h}{3}f_{u_i}(ih,x_i,u_i)
-\lambda_i g_{u_i}(ih,x_i,u_i)=0,
\quad i=2,4,\ldots,n-2,
\end{equation*}
\begin{equation*}
\frac{\partial J^*}{\partial {\lambda_i}}
=a_i-g(ih,x_i,u_i)=0,
\quad i=0,1,2,\ldots,n.
\end{equation*}
As it can be seen, the partial derivatives of $f$ and $g$ with respect 
to the elements of $X$ and $U$ are needed to provide the final system. 
Symbolic computation can be used to produce these partial derivatives.
Note that our method reduces the fractional optimal control problem 
\eqref{1.1}--\eqref{1.3} to the solution of a system of nonlinear 
algebraic equations. In some applications, solving the resulting system 
can be sensitive to the choice of initial values and may result
in a singular Jacobian matrix. In that case, some extra care in choosing
the initial values is needed.


\section{Error estimate}
\label{sec:4}

In this section, we present an error estimate for the method proposed 
in the previous section. To this end, first let us recall 
the following theorem from \cite{Nemati1}.

\begin{theorem}[See \cite{Nemati1}]
\label{th1} 
If a function $y(\cdot)\in C^3[0,t_f]$ is approximated 
using the set of modified hat functions by 
$y_n(t)=\sum\limits_{i=0}^ny(ih)\psi(t)$, then
\begin{equation*}
|y(t)-y_n(t)|=O(h^3).
\end{equation*}
\end{theorem}
According to Theorem~\ref{th1}, if the control function 
$u(\cdot)$ belongs to the space $C^3[0,t_f]$, 
then we will have
\begin{equation*}
|u(t)-u_n(t)|=O(h^3),
\end{equation*}
where $u_n(t)$ has been given in \eqref{3.3}. For obtaining 
an error estimate for the state function $x(\cdot)$, we suppose 
that $D^\alpha x(\cdot)\in C^3[0,t_f]$ and is approximated 
using \eqref{3.2}. Also, suppose that $D^T\Psi(t)$ is the 
approximation of function $\sum\limits_{i=0}^{m-1}q_i\frac{t^{i}}{i!}$. 
Then, using Theorem~\ref{th1}, we have
\begin{equation}
\label{4.6}
\begin{split}
&\left|D^\alpha x(t)-A^T\Psi(t)\right|=O(h^3),\\
&\left|\sum_{i=0}^{m-1}q_i\frac{t^{i}}{i!}-D^T\Psi(t)\right|=O(h^3).
\end{split}
\end{equation}
Using \eqref{2.1} and the initial conditions in \eqref{1.3}, we get
\begin{equation}
\label{4.5}
x(t)=I^{\alpha}(D^\alpha x(t))+\sum_{i=0}^{m-1}q_i\frac{t^{i}}{i!}.
\end{equation}
Now, we consider $x_n(t)=X^T\Psi(t)$ as an approximation of the state 
function $x(\cdot)$, where $X$ is given by \eqref{3.10}. Then, 
by neglecting the error of the operational matrix and utilizing 
\eqref{3.5}, \eqref{4.6}, and \eqref{4.5}, we obtain that
\begin{equation*}
\begin{split}
\left|x(t)-x_n(t)\right|
&=\left|I^\alpha (D^\alpha x(t))+\sum_{i=0}^{m-1}
q_i\frac{t^{i}}{i!}-A^TP^{\alpha}\Psi(t)-D^T\Psi(t)\right|\\
&\leq \left|I^\alpha( D^\alpha x(t))-A^TP^{\alpha}\Psi(t)\right|
+\left|\sum_{i=0}^{m-1}q_i\frac{t^{i}}{i!}-D^T\Psi(t)\right|\\
&=O(h^3).
\end{split}
\end{equation*}

To give an error estimate for the performance index obtained 
by the method we propose in Section~\ref{sec:3}, 
we make use of the following theorem.

\begin{theorem}[Composite Simpson's rule \cite{Burden}]
\label{th2}
Let $y(\cdot)\in C^4[0,t_f]$, $n$ be even, $h=t_f/n$, and $t_j=jh$, 
$j=0,1,\ldots,n$. Then, the composite Simpson's rule can be written 
with its corresponding error as
\begin{equation*}
\int_0^{t_f}y(t)dt=\frac{h}{3}\left[y(0)
+2\sum_{j=1}^{(n/2)-1}y(t_{2j})
+4\sum_{j=1}^{n/2}y(t_{2j-1})+y(t_f)\right]
-\frac{t_f}{180}h^4y^{(4)}(\mu),
\end{equation*}
where $\mu\in(0,t_f)$.
\end{theorem}

To complete the error discussion, we prove the following theorem. 

\begin{theorem}
Suppose that $f:\mathbb{R}^3\rightarrow \mathbb{R}$  
is sufficiently continuously differentiable and satisfies 
a Lipschitz condition with respect to the second argument. 
Let $x^*(\cdot)\in C^3[0,t_f]$ and $u^*(\cdot)\in C^3[0,t_f]$ 
be the optimal state and control functions, respectively, 
which minimize the performance index $J$ with optimal value 
$J_{opt}$. If $x_n(\cdot)$, $u_n(\cdot)$, and $J_n$ are the 
approximation of the optimal state function, control function, 
and performance index given by our method, then
\begin{equation*}
|J_{opt}-J_n|=O(h^3).
\end{equation*}
\end{theorem}

\begin{proof}
We put $t_j=jh$, $x_j:=x_n(t_j)$, and $u_j:=u(t_j)$. By setting 
$y(t):=f(t,x_n(t),u(t))$ and using Theorem~\ref{th2}, we have
\begin{equation}
\label{4.3}
\begin{split}
\int_0^{t_f}f(t,x_n(t),u(t))dt=\frac{h}{3}
&\left[f(0,x_0,u_0)+2\sum_{j=1}^{(n/2)-1}f(t_{2j},x_{2j},u_{2j})\right.\\
&\left.+4\sum_{j=1}^{n/2}f(t_{2j-1},x_{2j-1},u_{2j-1})+f(t_f,x_n,u_n)\right]+O(h^4)\\
&=f(\Theta,X,U)L+O(h^4).
\end{split}
\end{equation}
Since function $f$ satisfies a Lipschitz condition with respect 
to the second argument, there exists a real constant $C>0$ such that
\begin{equation}
\label{4.4}
|f(t,{\bf x}_1,u)-f(t,{\bf x}_2,u)|\leq C|{\bf x}_1-{\bf x}_2|.
\end{equation}
By assumption, we have
\begin{equation}
\label{4.1}
J_{opt}=\int_0^{t_f}f(t,x^*(t),u^*(t))dt
\end{equation}
and
\begin{equation}
\label{4.2}
J_n=f(\Theta,X,U)L.
\end{equation}
By subtracting \eqref{4.2} from \eqref{4.1} 
and using \eqref{4.3} and \eqref{4.4}, we get
\begin{equation}
\begin{split}
|J_{opt}-J_n|
&=\left|\int_0^{t_f}f(t,x^*(t),u^*(t))dt-f(\Theta,X,U)L\right|\\
&=\left|\int_0^{t_f}f(t,x^*(t),u^*(t))dt-\int_0^{t_f}f(t,x_n(t),u^*(t))dt\right.\\
&\left.+\int_0^{t_f}f(t,x_n(t),u^*(t))dt-f(\Theta,X,U)L\right|\\
&\leq \left|\int_0^{t_f}f(t,x^*(t),u^*(t))dt-\int_0^{t_f}f(t,x_n(t),u^*(t))dt\right|\\
&\quad +\left|\int_0^{t_f}f(t,x_n(t),u^*(t))dt-f(\Theta,X,U)L\right|\\
&\leq \int_0^{t_f}\left|f(t,x^*(t),u^*(t))-f(t,x_n(t),u^*(t))\right|dt+O(h^4)\\
&\leq \int_0^{t_f}C\left|x^*(t)-x_n(t)\right|dt+O(h^4)\\
&=O(h^3)+O(h^4)\\
&=O(h^3)
\end{split}
\end{equation}
The proof is complete.
\end{proof}


\section{Numerical solution of FOCPs with inequality constraints}
\label{sec:5}

In this section, we develop the method introduced in Section~\ref{sec:3} 
for solving constrained FOCPs. Consider problem \eqref{1.1}--\eqref{1.3} 
with an additional inequality constraint
\begin{equation}
\label{5.1}
H\left(t,x(t),D^{\alpha_1}x(t),\ldots,D^{\alpha_k} x(t),D^\alpha x(t),u(t)\right)
\leq 0.
\end{equation}
In order to solve this problem, we substitute the approximations given 
in \eqref{3.2}--\eqref{3.5} into \eqref{5.1} and get
\begin{equation*}\label{5.2}
\begin{split}
H\left(t,X^T\Psi(t),X_1^T\Psi(t),\ldots,X_k^T\Psi(t),A^T\Psi(t),U^T\Psi(t)\right)
\leq 0.
\end{split}
\end{equation*}
We set the following Newton--Cotes nodes as the collocation points:
\begin{equation*}
\tau_i=\frac{i+1}{2(n+1)}t_f,
\quad i=0,1,\ldots,2n.
\end{equation*}
Therefore, we get
\begin{equation}
\label{5.3}
H\left(\tau_i,X^T\Psi(\tau_i),X_1^T\Psi(\tau_i),
\ldots,X_k^T\Psi(\tau_i),A^T\Psi(\tau_i),U^T\Psi(\tau_i)\right) \leq 0, 
\quad i=0,1,\ldots,2n.
\end{equation}
Using \eqref{index}, \eqref{3.8}, and \eqref{5.3}, the constrained FOCP 
is reduced to the following nonlinear programming problem:
\begin{equation}
\label{5.4}
\begin{split}
\min~J_n&=f(\Theta,X,U)L\\
s.t.~&A^T-g(\Theta,X,X_1,\ldots,X_k,U)=0\\
&H\left(\tau_i,X^T\Psi(\tau_i),X_1^T\Psi(\tau_i),
\ldots,X_k^T\Psi(\tau_i),A^T\Psi(\tau_i),U^T\Psi(\tau_i)\right)\leq 0,
\quad i=0,1,\ldots,2n.
\end{split}
\end{equation}
It should be noted that the decision variables of the above nonlinear 
programming problem are: $a_i$ and $u_i$, $i=0,1,\ldots,n$. By employing 
a standard optimization solver, problem \eqref{5.4} can be easily solved
(see Section~\ref{sec:6}). After finding the optimal value of the decision 
variables, the approximations of the state and control functions are given, 
respectively, by \eqref{3.5} and \eqref{3.3}. Moreover, we get an approximation 
of the optimal performance index by substituting the results in \eqref{index}.


\section{Illustrative examples}
\label{sec:6}

In this section, we apply our method to some FOCPs 
and compare the results with the ones obtained by 
existing methods in the literature. To this aim, 
the $l_2$ norm of the error and the convergence order 
of state and control functions are defined, 
respectively, by
\begin{equation*}
E_n(x)=\left(\frac{1}{n}\sum_{i=1}^n(x(t_i)-x_n(t_i))^2\right)^{\frac{1}{2}},
\quad \epsilon_n(x)=\log_2\left(E_n(x)/E_{2n}(x)\right)
\end{equation*}
and
\begin{equation*}
E_n(u)=\left(\frac{1}{n}\sum_{i=1}^n(u(t_i)-u_n(t_i))^2\right)^{\frac{1}{2}},
\quad \epsilon_n(u)=\log_2\left(E_n(u)/E_{2n}(u)\right),
\end{equation*}
where $t_i=ih$, $x(t_i)$, and $u(t_i)$ are the exact state and control 
functions at $t_i$ and $x_n(t_i)$ and $u_n(t_i)$ are the obtained state 
and control functions by the proposed method at $t_i$.

In our implementation, the method was carried out using \textsf{Mathematica 11.3}.  
For solving the resulting systems of algebraic equations, the function 
\textsf{FindRoot} was used in two examples, which are non-constrained FOCPs. 
Moreover, in the case of constrained FOCP, the function \textsf{Minimize} 
was employed for solving the linear programming problem. In order to have a 
comparison with other methods, we report the CPU time (seconds) in some of the examples, 
which have been obtained on a 2.5 GHz Core i7 personal computer with 16~GB of RAM.


\begin{example}
\label{ex1}
As the first example, we consider the following FOCP 
borrowed from \cite{Salati}:
\begin{equation}
\label{6.1}
\begin{split}
\min~J
&=\int_0^{20}\left[1-\left(x(t)-0.01t^2-1\right)^2+u(t)-2\sqrt{\pi}J_0(4\sqrt{\pi})\right]^2dt\\
s.t. ~~&D^{0.5}x(t)=-\left(x(t)-0.01t^2-1\right)^2+u(t)+1+\frac{2}{75\sqrt{\pi}}t^{3/2},\\
&x(0)=1.
\end{split}
\end{equation}
The optimal value for the performance index is $J=0$, which is obtained by  
\begin{equation*}
\begin{split}
&u(t)=-\cos^2(4\sqrt{t})+2\sqrt{\pi}J_0(4\sqrt{t}),\\
&x(t)=\sin(4\sqrt{t})+0.01t^2+1,
\end{split}
\end{equation*}
where $J_0$ is the first kind Bessel function of order zero. 
This problem has been solved using the method we propose here,
with different values of $n$. By applying the method with $n$ 
equal to some powers of $2$, we report in Table~\ref{tab:1}
the error of our numerical results in the $l_2$ norm 
and the order of convergence. Table~\ref{tab:1} confirms 
that even when the state and control functions do not belong 
to $C^3[0,20]$ (as considered by Theorem~\ref{th1}), 
our method has high convergence order. 

In \cite{Salati}, problem \eqref{6.1} is considered 
with the boundary condition
\begin{equation*}
x(20)=5+\sin(8\sqrt{5}).
\end{equation*}
This problem is solved in \cite{Salati} by approximating the fractional 
integral by the Gr\"unwald--Letnikov formula (direct GL method), 
trapezoidal formula (direct TR method), and Simpson formula (direct SI method). 
A comparison between the results obtained by our method and the aforementioned 
methods given in \cite{Salati} is displayed in Table~\ref{tab:2}. The results 
show that our method has the same accuracy as the direct SI method. 
Furthermore, the approximate and exact state and control functions obtained 
by the present method with $n=50$ and $n=100$, together with the corresponding 
error functions, are plotted in Figure~\ref{fig:1}.
\begin{table}[!ht]
\centering
\caption{Example~\ref{ex1}: error and convergence order 
for the state and control functions.}\label{tab:1}
\begin{tabular}{lllllll} \hline
&  &\multicolumn{2}{c}{$x(t)$} &&\multicolumn{2}{c}{$u(t)$}  \\
\cline{3-4}\cline{6-7}
$n$ &   & $E_n(x)$ &$\epsilon_n(x)$ && $E_n(u)$& $\epsilon_n(u)$\\
\hline
$8$ & & $1.23e+0$ & $2.34$ && $3.10e+0$ & $3.63$\\
$16$ &   & $2.43e-1$& $3.09$  && $2.51e-1$ & $3.56$\\
$32$ &   &$2.86e-2$ &$3.41$& &$2.13e-2$ & $2.44$\\
$64$ &   &$2.68e-3$ &$3.50$&& $3.92e-3$  & $3.37$\\
$128$ &   &$2.36e-4$& $3.52$ & & $3.79e-4$ & $3.57$\\
$256$ &   &$2.06e-5$ &---&& $3.18e-5$  & ---\\
\hline
\end{tabular}
\end{table}
\begin{table}[!ht]
\scriptsize
\centering
\caption{Example~\ref{ex1}: comparison of the numerical results 
obtained by the present method versus methods of \cite{Salati}.}\label{tab:2}
\begin{tabular}{llllllllllll} \hline
&  \multicolumn{2}{c}{Direct GL method \cite{Salati}} 
&&\multicolumn{2}{c}{Direct TR method \cite{Salati}}
&&\multicolumn{2}{c}{Direct SI method \cite{Salati}} 
&& \multicolumn{2}{c}{Present method} \\
\cline{2-3}\cline{5-6}\cline{8-9}\cline{11-12}
$n$ &    $E_n(x)$ &$E_n(u)$ && $E_n(x)$&$E_n(u)$&&$E_n(x)$&$E_n(u)$&& $E_n(x)$&$E_n(u)$\\ \hline
$100$ &  $1.11e-1$ & $1.68e-1$ && $1.48e-2$ & $2.07e-2$&& $5.60e-4$ & $8.99e-4$ &&$5.63e-4$ &$9.03e-4$ \\
$200$ &  $5.71e-2$ & $9.19e-2$ && $3.71e-3$ & $5.21e-3$&& $4.91e-5$ & $7.66e-5$ &&$4.92e-5$ &$7.68e-5$ \\
$300$ &  $3.94e-2$ & $6.37e-2$ && $1.65e-3$ & $2.32e-3$&& $1.18e-5$ & $1.80e-5$ &&$1.18e-5$ &$1.80e-5$ \\ 
\hline
\end{tabular}
\end{table}
\begin{figure}[!ht]
\centering
\includegraphics[scale=0.48]{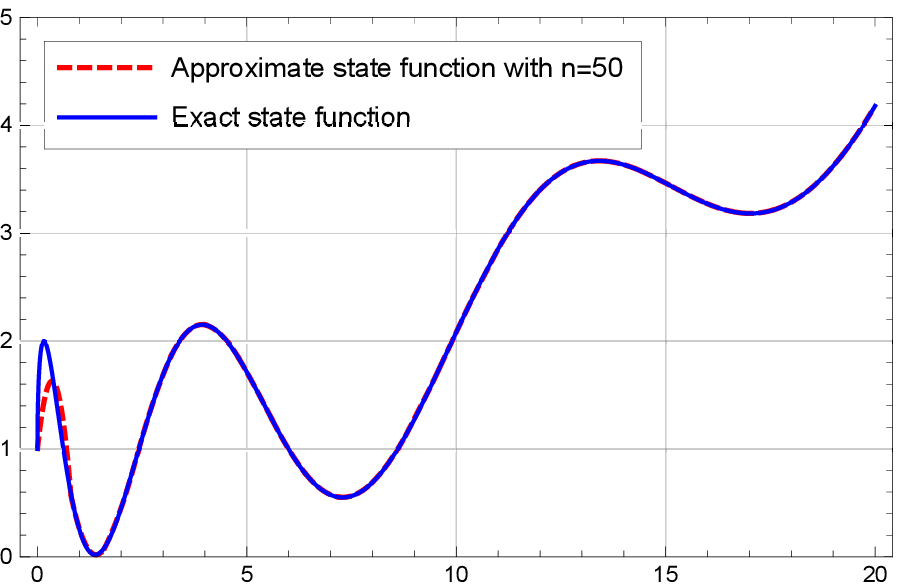}
\includegraphics[scale=0.48]{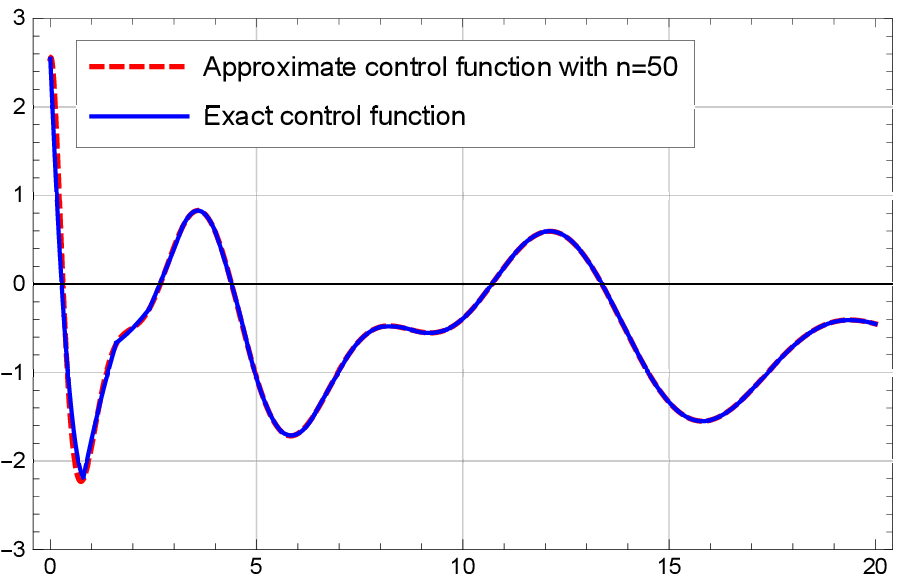}
\includegraphics[scale=0.49]{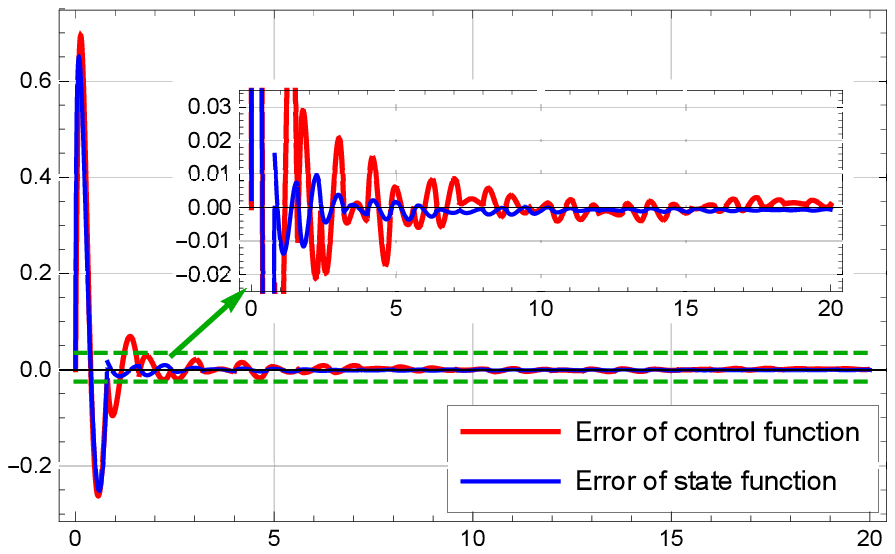}
\includegraphics[scale=0.48]{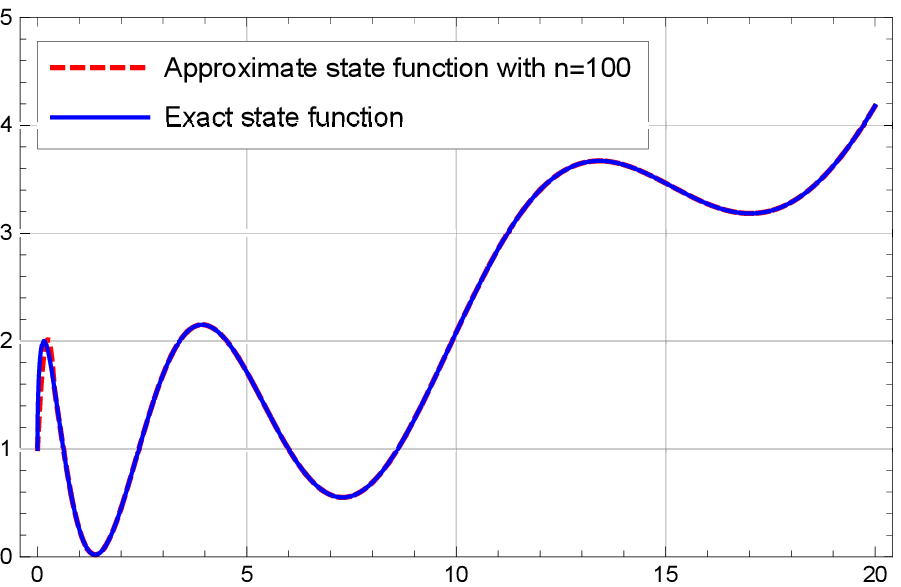}
\includegraphics[scale=0.48]{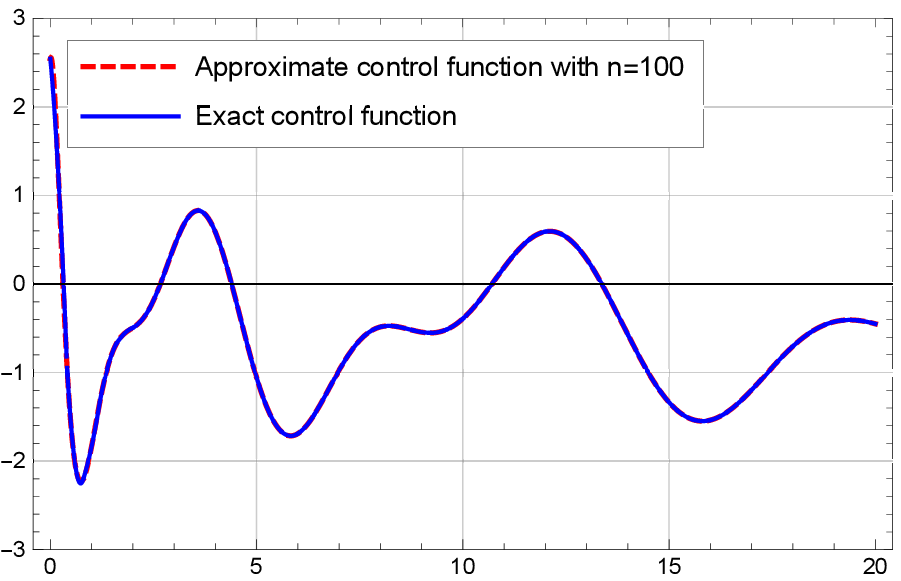}
\includegraphics[scale=0.49]{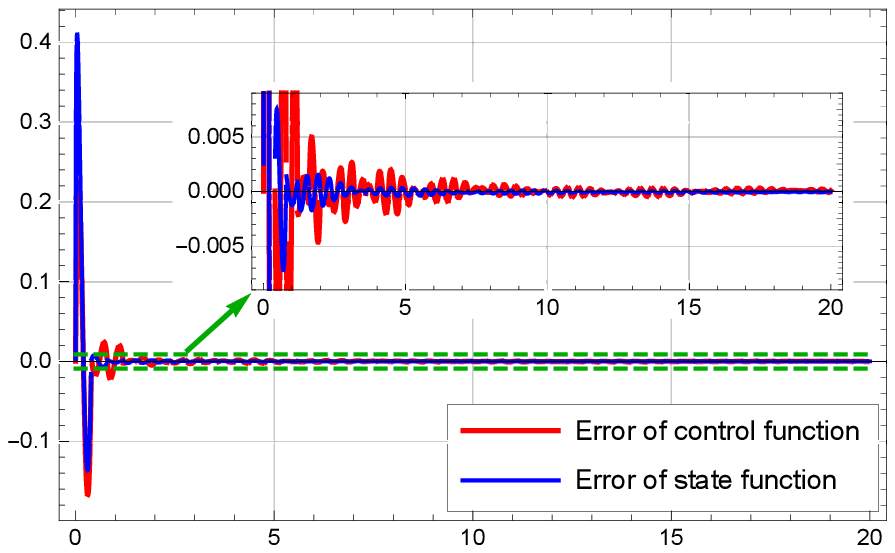}
\caption{Example~\ref{ex1}: comparison of the exact and numerical 
solutions along with error functions. Above: $n=50$. Bellow: $n=100$.}
\label{fig:1}
\end{figure}
\end{example}


\begin{example}
\label{ex2}
Consider the following constrained FOCP \cite{Alipour,Mashayekhi,Yonthanthum}:
\begin{equation}
\label{6.2}
\begin{split}
\min~&J=\int_0^1(-\ln 2)x(t)dt\\
s.t.~&D^\alpha x(t)=(\ln 2)(x(t)+u(t)),\\
&x(0)=0,\\
&\left|u(t)\right|\leq 1,\\
&x(t)+u(t)\leq 2.
\end{split}
\end{equation}
In this problem, we have $t_f=1$. By taking $n=2$,  
the value of $h$ is given as $h=\frac{1}{2}$. We set
\begin{equation}
\label{6.3}
D^\alpha x(t)=A^T\Psi(t),\quad u(t)=U^T\Psi(t),
\end{equation}
where
\begin{equation*}
A=[a_0,a_1,a_2]^T,\quad U=[u_0,u_1,u_2]^T,
\quad \Psi(t)=[\psi_0(t),\psi_1(t),\psi_2(t)]^T,
\end{equation*}
with unknown parameters $a_i$, $u_i$, $i=0,1,2$, and
\begin{equation*}
\psi_0(t)=\left\{\begin{array}{ll} 2t^2-3t+1,&0\leq t\leq 1,\\
&\\
0,&\text{otherwise},
\end{array}\right.
\end{equation*}
\begin{equation*}
\psi_1(t)=\left\{\begin{array}{ll}-4t^2+4t,&0\leq t\leq 1,\\
&\\
0,&\text{otherwise},
\end{array}\right.
\end{equation*}
\begin{equation*}
\psi_2(t)=\left\{\begin{array}{ll}2t^2-t,&0\leq t\leq 1,\\
&\\
0,&\text{otherwise}.
\end{array}\right.
\end{equation*}
Taking into account the initial condition 
in \eqref{6.2}, we obtain that
\begin{equation}
\label{6.4}
x(t)=A^TP^{(\alpha)}\Psi(t),
\end{equation}
where
\begin{equation*}
P^{(\alpha)}=\left[\begin{array}{ccc}
0 & p_{01} & p_{02}\\
0 & p_{11} & p_{12}\\
0 & p_{21} & p_{22}\\
\end{array}
\right].
\end{equation*}
By substituting \eqref{6.4} into the performance index $J$ 
and using \eqref{index}, we get
\begin{equation}
\label{6.5}
J=\frac{-\ln 2}{6}\left(4 a_0 p_{01}+a_0 p_{02}
+4 a_1 p_{11}+a_1 p_{12}+4 a_2 p_{21}+a_2 p_{22}\right).
\end{equation}
Moreover, using \eqref{6.3} and \eqref{6.4}, 
the dynamic system in \eqref{6.2} is reduced to
\begin{equation*}
A^T-(\ln 2)(A^TP^{(\alpha)}+U^T)=0,
\end{equation*}
which gives the following system of equations:
\begin{equation}\label{6.6}
\begin{split}
&a_0- \ln 2 u_0=0,\\
&a_1-\ln 2 \left(a_0 p_{01}+a_1 p_{11}+a_2 p_{21}+u_1\right)=0,\\
&a_2-\ln 2 \left(a_0 p_{02}+a_1 p_{12}+a_2 p_{22}+u_2\right)=0.\\
\end{split}
\end{equation}
Finally, we take into account the inequality constraints in \eqref{6.2} 
at $\tau_i=\frac{i+1}{6}$, $i=0,1,\ldots,4$, to obtain
\begin{equation}
\label{6.7}
\begin{split}
&\left|\frac{5 u_0}{9}+\frac{5 u_1}{9}-\frac{u_2}{9}\right|\leq 1,\\
&\left|\frac{2 u_0}{9}+\frac{8 u_1}{9}-\frac{u_2}{9}\right|\leq 1,\\
&\left|u_1\right|\leq 1,\\
&\left|-\frac{u_0}{9}+\frac{8 u_1}{9}+\frac{2 u_2}{9}\right|\leq 1,\\
&\left|-\frac{u_0}{9}+\frac{5 u_1}{9}+\frac{5 u_2}{9}\right|\leq 1,\\
&\frac{1}{9} \left(5 a_0 p_{01}-a_0 p_{02}+5 a_1 p_{11}
-a_1 p_{12}+5 a_2 p_{21}-a_2 p_{22}+5 u_0+5 u_1-u_2\right)\leq 2,\\
&\frac{1}{9} \left(8 a_0 p_{01}-a_0 p_{02}+8 a_1 p_{11}
-a_1 p_{12}+8 a_2 p_{21}-a_2 p_{22}+2 u_0+8 u_1-u_2\right)\leq 2,\\
&a_0 p_{01}+a_1 p_{11}+a_2 p_{21}+u_1\leq 2,\\
&\frac{1}{9} \left(2 \left(4 a_0 p_{01}+a_0 p_{02}
+4 a_1 p_{11}+a_1 p_{12}+4 a_2 p_{21}+a_2 p_{22}
+4 u_1+u_2\right)-u_0\right)\leq 2,\\
&\frac{1}{9} \left(5 \left(a_0 p_{01}+a_0 p_{02}
+a_1 p_{11}+a_1 p_{12}+a_2 p_{21}+a_2 p_{22}+u_1
+u_2\right)-u_0\right)\leq 2.
\end{split}
\end{equation}
In summary, the problem is reduced to the minimization of \eqref{6.5} 
subject to the conditions given in \eqref{6.6} and \eqref{6.7}. 

In the particular case when $\alpha=1$, the exact solution 
of the problem is given by
\begin{equation*}
x(t)=2^t-1,\quad u(t)=1,
\end{equation*}
and the optimal value of the performance index with seven significant digits 
is $J=-0.3068528$. By considering $\alpha=1$, the operational matrix 
of integration is given by
\begin{equation*}
P^{(1)}=\left[
\begin{array}{ccc}
 0 & \frac{5}{24} & \frac{1}{6} \\
 &&\\
 0 & \frac{1}{3} & \frac{2}{3} \\
 &&\\
 0 & -\frac{1}{24} & \frac{1}{6} \\
\end{array}
\right].
\end{equation*}
Therefore, by substituting the elements of $P^{(1)}$ into \eqref{6.5}--\eqref{6.7}, 
the following linear programming problem is produced:
\begin{eqnarray*}
\min~&\frac{-\ln 2}{6}\left(a_0+2 a_1\right),\\
s.t.~&a_0-u_0 \ln 2=0,\\
&a_1-\ln 2 \left(\frac{5 a_0}{24}+\frac{a_1}{3}-\frac{a_2}{24}+u_1\right)=0,\\
&a_2-\ln 2 \left(\frac{a_0}{6}+\frac{2 a_1}{3}+\frac{a_2}{6}+u_2\right)=0,\\
&\left|\frac{5 u_0}{9}+\frac{5 u_1}{9}-\frac{u_2}{9}\right|\leq 1,\\
&\left|\frac{2 u_0}{9}+\frac{8 u_1}{9}-\frac{u_2}{9}\right|\leq 1,\\
&\left|u_1\right|\leq 1,\\
&\left| -\frac{u_0}{9}+\frac{8 u_1}{9}+\frac{2 u_2}{9}\right|\leq 1,\\
&\left|-\frac{u_0}{9}+\frac{5 u_1}{9}+\frac{5 u_2}{9}\right|\leq 1,\\
&\frac{1}{72} \left(7 a_0+8 a_1-3 a_2+40 u_0+40 u_1-8 u_2\right)\leq 2,\\
&\frac{1}{18} \left(3 a_0+4 a_1-a_2+4 u_0+16 u_1-2 u_2\right)\leq 2,\\
&\frac{1}{24} \left(5 a_0+8 a_1-a_2+24 u_1\right)\leq 2,\\
&\frac{1}{9} \left(2 a_0+4 a_1-u_0+8 u_1+2 u_2\right)\leq 2,\\
&\frac{1}{72} \left(15 a_0+40 a_1+5 a_2-8 u_0+40 u_1+40 u_2\right)\leq 2.
\end{eqnarray*}

By solving this problem, we find
\begin{equation*}
a_0=0.6931472,~a_1=0.9795332,~a_2=1.3859775, u_0=u_1=u_2=1.
\end{equation*}
Substituting $u_i=1$, $i=0,1,2$, into \eqref{6.3}, the exact control function is obtained.
The approximate and exact state functions are plotted in Figure~\ref{fig:2}. 
In this case, the approximate value of the performance index is $J_2=-0.3063957$.

We have employed our method for solving problem \eqref{6.2} with different values of $n$. 
For all the considered values of $n$, the method gives the exact control function. 
Table~\ref{tab:3} displays the numerical results for $x(t)$, $\epsilon_n(x)$, $J$, 
and the CPU time consumed for solving the resulting linear programming problem. 
It is seen that the present method gives the exact solution of $J$, with seven 
significant digits, with $n=32$ and that the convergence order of the state 
function is $O(h^4)$. In Table~\ref{tab:3}, the method of \cite{Yonthanthum} 
refers to the hybrid block-pulse and Taylor polynomials method, the
method of \cite{Mashayekhi} denotes the method of hybrid block-pulse 
and Bernoulli polynomials, and the method of \cite{Alipour} refers 
to the Bernstein polynomials method. 
\begin{figure}[!ht]
\centering
\includegraphics[scale=1]{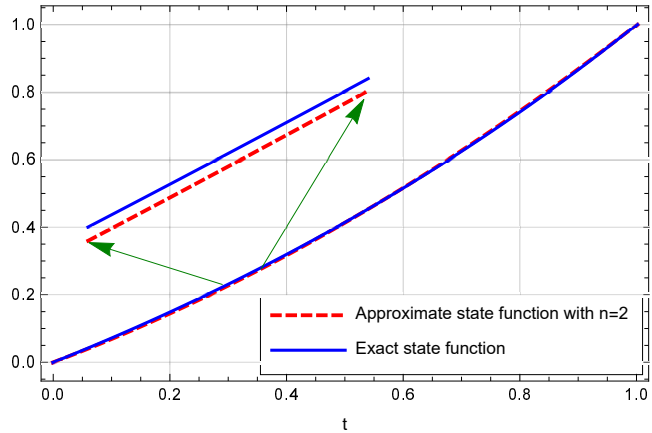}
\caption{Example~\ref{ex2}: comparison of the exact and numerical 
state functions with $n=2$ and $\alpha=1$.}\label{fig:2}
\end{figure}
\begin{table}[!ht]
\scriptsize
\centering
\caption{Numerical results for Example~\ref{ex2}.}\label{tab:3}
\begin{tabular}{llcccc}
\hline
Methods & &$E_n(x)$&$\epsilon_n(x)$& $J$ & CPU Time  \\ \hline
Present   &$n=2$ &$8.07e-4$ & $4.02$&$-0.3063957$&  $0.015$\\
method    &$n=4$ & $4.99e-5$& $4.01$ &$-0.3068248$ & $0.015$\\
          &$n=8$& $3.09e-6$ &  $4.01$ & $-0.3068511$& $0.015$\\
          &$n=16$& $1.92e-7$ &  $4.00$ &$-0.3068527$ &$0.032$\\
          &$n=32$& $1.20e-8$ &   ---    & $-0.3068528$&$0.032$\\ \hline
Method of \cite{Yonthanthum}&$M=3,$ $N=1$  & &&$-0.30683$ & $40.091$\\ \hline
Method of \cite{Mashayekhi} & $M=3,$ $N=1$& && $-0.30684$ & $53.133$\\ \hline
Method of \cite{Alipour} & $M=3$ & && $-0.30685$ &\\ \hline
\end{tabular}
\end{table}
\end{example}


\begin{example}
\label{ex3}
Let us now consider the following nonlinear FOCP \cite{Lotfi,Keshavarz}:
\begin{equation*}
\begin{split}
\min~&J=\int_0^1\left[ e^t\left(x(t)-t^4+t-1\right)^2+(1+t^2) \left(u(t) + 1- t + t^4 -\frac{8000t^{\frac{21}{10}}}{77\Gamma(\frac{1}{10})}\right)^2\right]dt,\\
s.t.~&D^{1.9}x(t)=x(t)+u(t),\\
&x(0)=1,\quad x'(0)=-1.
\end{split}
\end{equation*}
The state and control functions that minimize the performance index $J$ are given by 
\begin{equation*}
\begin{split}
&x(t)=1-t+t^4,\\
&u(t)= -1+ t - t^4 +\frac{8000}{77\Gamma(\frac{1}{10})}t^{\frac{21}{10}},
\end{split}
\end{equation*}
with optimal value $J_{opt}=0$. By applying our method with different values of $n$, 
we obtain the numerical results displayed in Table~\ref{tab:4}. In this table, 
we give a comparison between the results for $J$ and CPU times consumed 
for solving the final algebraic system resulted by our method, 
the method of \cite{Lotfi}, which uses Legendre polynomials as the 
basis functions and the Gauss quadrature rule for integrating the FOCP, 
and also the method of \cite{Keshavarz}, which employs Bernoulli polynomials 
and the Gauss quadrature rule for integrating the FOCP. In \cite{Lotfi} 
and \cite{Keshavarz}, $m$ refers to the maximum order of the
Legendre and Bernoulli polynomials, respectively, that have been used 
for approximating the state and control functions. It can be seen from 
Table~\ref{tab:4} that the method we suggest here approximates 
the state and control functions with convergence of order $O(h^3)$. 
In Figure~\ref{fig:3} (left), the results for $E_n(x)$, $E_n(u)$, 
and $E_n(J)=|J_{opt}-J_n|$, obtained by employing our method 
for some selected values of $n$, are plotted in a logarithmic scale. 
This figure confirms the $O(h^3)$ accuracy order of the state and control functions. 
It also shows that the accuracy order of the performance index is higher 
than $O(h^3)$. Moreover, the CPU times of the method are plotted 
in Figure~\ref{fig:3} (right).
\begin{table}[!ht]
\scriptsize
\centering
\caption{Numerical results for Example~\ref{ex3}.}\label{tab:4}
\begin{tabular}{llcccclc} \hline
Methods & &$E_n(x)$&$\epsilon_n(x)$&$E_n(u)$&$\epsilon_n(u)$& $J$ & CPU Time  \\
\hline
Present   &$n=4$ & $7.10e-4$& $3.39$ &$2.98e-4$ & $3.03$&$9.64314e-7$ & $0.000$\\
method    &$n=8$& $6.75e-5$ & $3.33$ &$3.65e-5$ & $3.15$& $1.00418e-8$& $0.015$\\
          &$n=16$& $6.69e-6$& $3.28$ &$4.10e-6$ & $3.18$&$1.06677e-10$ &$0.046$\\
          &$n=32$& $6.91e-7$&  $3.22$    &$4.52e-7$ & $3.17$& $1.19487e-12$&$0.109$\\
          &$n=64$& $7.42e-8$&   $3.18$   &$5.03e-8$ & $3.15$& $1.41601e-14$&$0.765$\\
          &$n=128$& $8.20e-9$&  $3.15$    &$5.66e-9$ & $3.14$& $1.75827e-16$&$4.922$\\
          &$n=256$& $9.24e-10$&   ---    &$6.44e-10$ & ---& $2.25012e-18$&$21.282$\\ \hline
Method of \cite{Lotfi} & $m=4$   & & &&&$5.42028e-7$ & $0.141$\\
&$m=8$&&&&&$8.22283e-10$&$0.296$\\ \hline
Method of \cite{Keshavarz} & $m=4$ & && && $5.16864e-7$ &$0.078$\\
&$m=8$&&&&&$4.23025e-11$&$0.171$\\ \hline
\end{tabular}
\end{table}
\begin{figure}[!ht]
\centering
\includegraphics[scale=0.7]{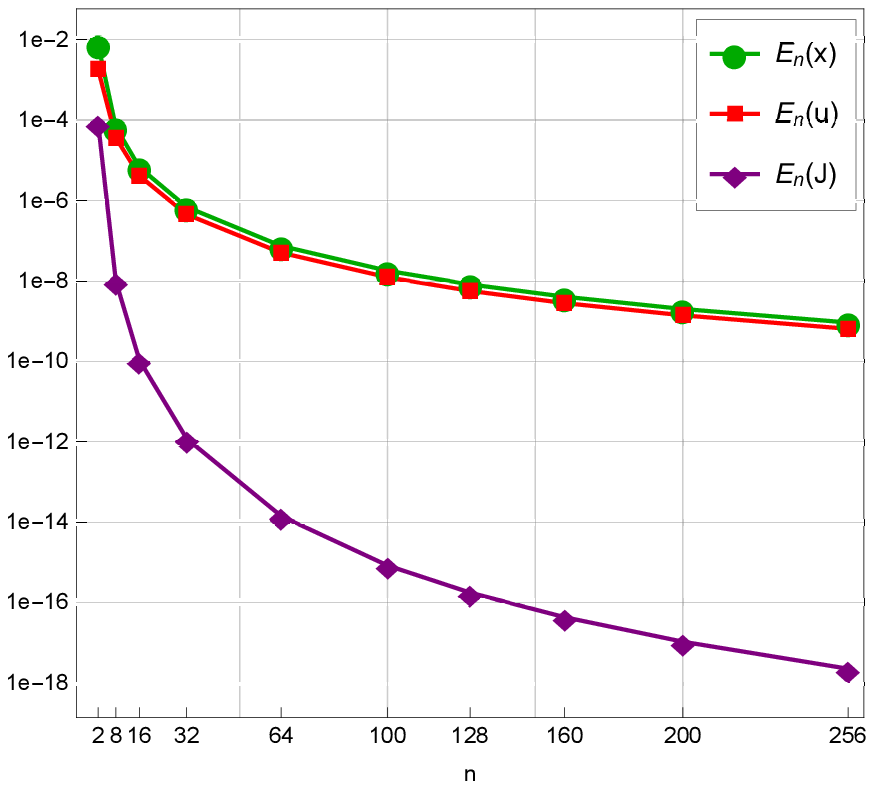}
\includegraphics[scale=0.7]{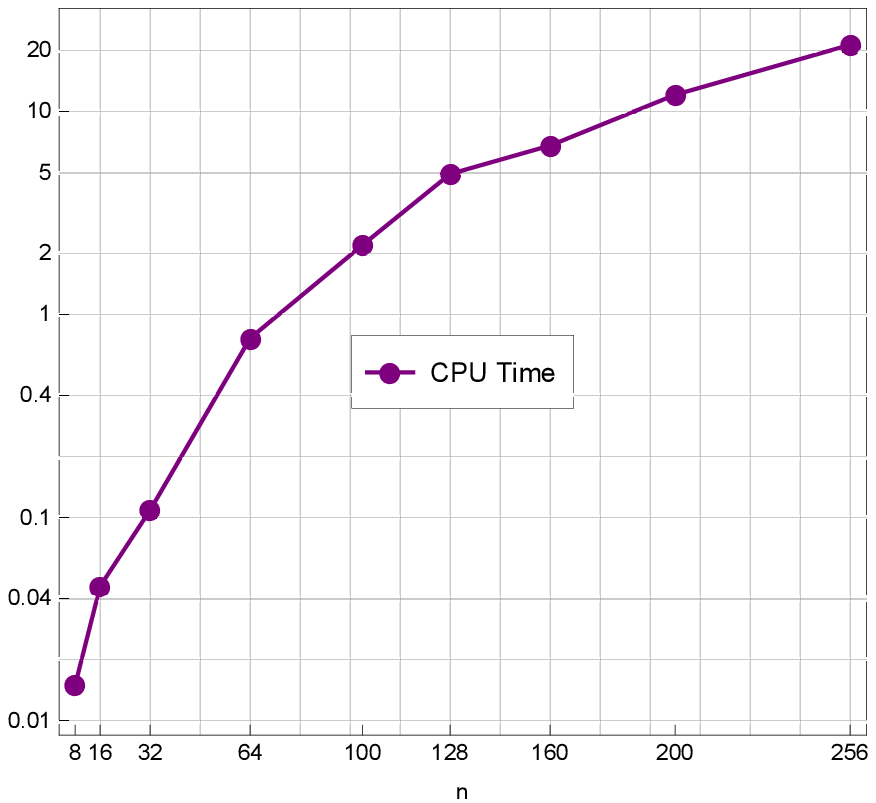}
\caption{Numerical results of Example~\ref{ex3} 
for some selected $n$: the measured $E_n(x)$, 
$E_n(u)$, and $E_n(J)$ in logarithmic scale (left); 
CPU time (right)}\label{fig:3}
\end{figure}
\end{example}


\section{Conclusion}
\label{sec:7}

We introduced a new numerical approach for solving fractional 
optimal control problems (FOCPs). Our scheme uses modified hat 
functions as basis functions and gives approximations of the control 
and state variables as linear combinations of such basis functions. 
The properties of the basis functions, together with those of the
Caputo derivative and Riemann--Liouville integral, allow us to 
reduce the FOCP to a system of nonlinear algebraic equations, 
which greatly simplifies the problem. An error estimate of 
the method is proved. Moreover, a generalization  
is given for solving constrained FOCPs. The method is employed 
for solving three illustrative examples and the obtained results 
confirm the efficiency, accuracy, and high performance of our technique
when compared with state of the art numerical schemes available
in the literature. We claim that our method can be very useful in real applications.
We are currently investigating its application to human respiratory 
syncytial virus infection and to the model introduced in \cite{MR3872489}.
Another direction of future research consists
to generalize our method to variable-order FOCPs,
which is an area under development \cite{MR3822307,MR3885740}. 


\section*{Acknowledgments}

Lima acknowledges support from 
Funda\c{c}\~ao para a Ci\^encia e a Tecnologia
(FCT, the Portuguese Foundation for Science
and Technology) through project UID/MAT/04621/2019;
Torres was supported by FCT
within project UID/MAT/04106/2019 (CIDMA).
The authors are grateful to an anonymous reviewer 
for suggestions to improve the article.



\end{document}